\documentclass[11pt, epsfig]{article}
\usepackage{epsfig, amsmath, amssymb, amsthm, times}

\usepackage[T1]{fontenc}
\usepackage{hyperref}

\def\la{\lambda}
\def\al{\alpha}

\newcommand{\pa}[1]{{\color{red} {\bf [}{\color{blue} P: } #1{\bf ]}}}

\usepackage[dvipsnames]{xcolor}
\usepackage{amsmath,amsthm,amssymb,amscd,setspace,
	parskip}
\usepackage{nccmath}
\usepackage{graphicx}
\usepackage{epstopdf}
\usepackage{marvosym}
\usepackage[flushleft]{threeparttable}
\usepackage{caption}
\usepackage{booktabs}
\usepackage{blkarray}

\usepackage{algorithm}
\usepackage{blkarray}
\usepackage{algpseudocode}
\usepackage{calc}
\usepackage{capt-of}
\usepackage[blocks]{authblk}
\usepackage{xcolor}
\usepackage{subcaption}
\usepackage{mwe}
\usepackage{multirow}
\usepackage{dsfont}
\newcommand{\bel}{\begin{equation}\label}
	\newcommand{\ee}{\end{equation}}

\def\PP{\mathcal P}
\def\D{\Delta}
\def\R{{\mathbb R}}
\def\C{{\mathbb C}}

\def\Z{{\mathbb Z}}

\def\<{\langle}
\def\>{\rangle}
\def\P{{\mathds P}}

\def\E{{\mathds E}}
\def\eps{\epsilon}

\def\i{\underline i}

\def\II{\mathbb I}
\def\0{\underline 0}
\def\1{\underline 1}
\def\lf{\lfloor}
\def\rf{\rfloor}

\newtheorem{theorem}{Theorem}[section]
\newtheorem{proposition}[theorem]{Proposition}

\newtheorem{lemma}[theorem]{Lemma}
\newtheorem{remark}{Remark}[section]
\theoremstyle{definition}

\title{Expected number of jumps and the number of active particles in TASEP}
\author{Paweł Hitczenko \thanks{Department of Mathematics, Drexel University, Philadelphia, PA 19130, USA \\ Email: \texttt{pawel.hitczenko@drexel.edu}}  \hspace{1mm} and Jacek Wesołowski \thanks{Faculty of Mathematics and Information Sciences, Warsaw University of Technology, Koszykowa 75,
00-662 Warsaw, Poland \\
Email: \texttt{jacek.wesolowski@pw.edu.pl}}}

\begin{document}

	\maketitle
	
	\begin{abstract}
	    For a TASEP on $\mathbb Z$ with the step initial condition we identify limits as $t\to\infty$ of the expected total number of jumps until time $t>0$ and the expected number of active particles at a time $t$. We also connect the two quantities proving that non-asymptotically, that is as a function of  $t>0$, the latter is the derivative of the former. Our approach builds on asymptotics derived by Rost and intensive use of the fact that the rightmost particle evolves according to the Poisson process.
	\end{abstract}
	
	\section{Introduction}

 Totally asymmetric simple exclusion process (TASEP) is among the most investigated interacting particle systems. At any time there is at most one particle at each of the sites of a one dimensional lattice $\mathbb{Z}$ and the dynamics of the model is as follows: each particle waits an $\rm{Exponential}(1)$ time (which we will refer to as a 'clock') and then jumps to the site to its right, provided that there is no particle at that site. If that site is occupied, the particle does not move and its clock resets.  The $\rm{Exponential}(1)$ clocks  for the same particle and the clocks corresponding to different particles are  all mutually  independent. In this paper, we consider a TASEP with the step initial condition, i.e. initially particles are placed at the sites $0, -1,-2,\dots$ and the sites $1, 2, 3,\dots $ are unoccupied.
	
	TASEP is the simplest of the interacting particle models, referred to as simple exclusion processes. They were introduced in mathematics in the late 1960s by Spitzer~\cite{s} and have been studied ever since. Aside from the early work by Spitzer, followed by Liggett's \cite{l}, later contributions by Rost \cite{r} and Sch\"utz \cite{sc} seem to have been quite influential. In particular, the latter paper gives a determinantal formula for the probability that there are particles at any finite set of sites in $\mathbb Z$ at a time $t$, given their positions at a time $0$,  while the former is concerned with the asymptotic properties of the number of particles in a given region.  Rost's results were further elaborated on e.g. in \cite{sep}.
	
	In recent years, there has been substantial interest in TASEP as it has been identified as one of the simplest models in a universality class of stochastic growth models referred to as the KPZ universality class. In this context, the so-called height function and its fluctuations have been the main objects of the studies, see, for example, \cite{bf,bfs,cfp,j,ps} where TASEP with the step initial condition (or its generalizations) is considered. 
    A more detailed discussion of TASEP and its role in the KPZ universality class may be found in \cite{a, mqr, mr} and references therein.  In particular, in \cite[Theorem~3.13]{mqr}  TASEP is used to construct  the so--called  KPZ fixed point (see also \cite[Section~16]{dv} where this is further discussed and convergence of TASEP to the directed landscape is established).
	
	
	As we mentioned earlier, most of the up--to--date  work on TASEP concerns the height function and its fluctuations.  For example, for  the step initial condition it was shown in \cite{j} that the scaling limit of the height function is the Tracy--Widom largest eignevalue distribution for the  Gaussian Unitary Ensemble.
   
	In this paper, we will focus on two aspects of TASEP that seem basic to describing its evolution, though we were not able to trace related results in the literature: they concern the total number of jumps of all particles up to a time $t$ and the number of active particles at the time $t$ (i.e., the particles that can move at a time $t$). Specifically, we will establish the asymptotics of the expected values of these quantities (after a suitable normalization).  We refer to Theorem~\ref{pole} and Theorem~\ref{thm_act_part} below for the precise statements. 
	
	It is natural to expect that these two quantities are related: how fast the process is evolving at a given time should depend on how many particles can move at that time. Nonetheless, we have not seen any such formal statements in the literature.

 For  TASEP with the step initial configuration we will formalize this relationship. Namely, we will prove in  Theorem~\ref{thm_active} below that the number of active particles at a time $t$ is the derivative of the  number of jumps by the process up to that time, at least at the level of the expected values.  It is conceivable that similar statement is true for more general classes of interacting particle processes.
 
	
	Our methods are elementary. We rely heavily on the fact that the jumps of the rightmost particle are unconstrained and thus follow a Poisson process. We will also exploit some of the techniques and results established by Rost's paper \cite{r} that heavily influenced our work.

	\section{Notation and Preliminaries}

	Following Rost~\cite{r} we let $X(k,t)$ be the indicator of the event that there is a particle in a position $k$ at time $t$, and we let
	\[S(k,t):=\sum_{i>k}X(i,t)\]
	be the number of particles to the right of $k$ at time $t$.  
	For step initial conditions, limiting behavior of the distribution of $S(\lf ut\rf,t)$ for fixed $u\in[0,1)$ and $t\to\infty$  was determined in \cite{j}, see Corollary~1.7 there.
	For our purpose, the following estimate for the expected value of $S(k,t)$ will be sufficient.

When we refer to particles, we will  number them from the right to the left, i.e. the rightmost particle will be referred to as the first, the second rightmost as the second, etc.

	\begin{lemma}\label{lem_exp_skt} For $k\in \mathbb Z_+$ and $t>0$ we have
		\bel{exp_skt} \E\ S(k,t)\le\frac{t^{k+1}}{k!}.\ee
	\end{lemma}
	\proof
	We first write
	\[\E\ S(k,t)=\sum_{m\ge1}\P(S(k,t)\ge m).\]
	Next, note that $S(k,t)\ge m$ means that there are at least $m$ particles to the right of $k$.  This means, in particular, that the $m$th particle made at least $m-1+k+1=k+m$ jumps by the time $t$. Hence, the first particle must have made at least as many moves by the time $t$. In other words \bel{fst_particle}
	\mathds P(S(k,t)\ge m)\le \mathds P(N_t^{(1)}\ge k+m),
	\ee
	where $N_t^{(1)}$ is the number of jumps by the first particle up to time $t$. Since $N_t^{(1)}$ is ${\rm Poisson}(t)$ random variable, we further have 
		\begin{align*}\sum_{m\ge1}\,\mathds P\left(N_t^{(1)}\ge m+k\right)&=\sum_{m\ge1}e^{-t}\sum_{j\ge0}\frac{t^{m+k+j}}{(m+k+j)!}\le\frac{t^{k}}{k!}\sum_{m\ge1}e^{-t}\sum_{j\ge0}\frac{t^{m+j}}{(m+j)!}\\&=\frac{t^{k}}{k!}\sum_{m\ge1}\,\mathds P(N_t^{(1)}\ge m)
			\end{align*}
	where for the inequality we used  $(m+k+j)!\ge(m+j)!k!$. 
	The assertion follows since
	\[\sum_{m\ge1}\,\mathds P(N_t^{(1)}\ge m)=\mathds E\, N_t^{(1)}=t.\]
	\endproof

	\section{Total number of jumps}

	Let $P_t$ denote the total number of jumps of (all) particles in a TASEP until $t>0$, which corresponds to the area below the height function representation of the TASEP. In this section we establish the following asymptotics for $\E\, P_t$.

	\begin{theorem} \label{pole}As $t\to\infty$ 
		\[\frac{\mathds E\, P_t}{t^2}\to\frac16.\]
	\end{theorem}
	\proof

	To facilitate the proof, we represent $P_t$ as follows.
	For $t\ge0$ we visualize the graph of the function $x\mapsto S(\lf x\rf,t)$ as columns of squares of unit size on the upper half-plane whose heights are decreasing (by no more than $1$) with $x$ (see Figure~\ref{fig_Skt}). When a $j$th particle jumps a new square is added at the right end of the $j$th row. With this interpretation, the total number of jumps by the process
	up to time $t$ is the number of added squares, i.~e., the area between the graphs of $S(\lf \cdot\rf,t)$ and $S(\lf \cdot\rf,0)$.

	That is,
	\bel{pt}P_t:=\sum_{k\in \Z}(S(k,t)-S(k,0)).\ee

	Consequently,
		\begin{align*}
		\mathds E\,P_t&=\E\,\sum_{k\in \Z}(S(k,t)-S(k,0))=\int_\R\,\mathds E\,(S(\lf v\rf,t)-S(\lf v\rf ,0))dv\\ &=t\int_\R\,\mathds E\,(S(\lf ut\rf,t)-S(\lf ut\rf ,0))du=t^2\int_\R\,\mathds E\,\left(\frac{S(\lf ut\rf,t)-S(\lf ut\rf ,0)}t\right)du.
	\end{align*}
	
	\begin{figure} 
		
		\begin{picture}(100,200)(120,20)
			
			\thinlines
			\color{black}
			\put(340,50){\vector(0,10){120}}
			\put (450,50){\vector(4,0){10}}
			\put(400,50){\line(0,10){20}}
			\put(380,50){\line(0,10){40}}
			\put(360,50){\line(0,10){40}}
			\put(320,50){\line(0,10){40}}
			\put(300,50){\line(0,10){60}}
			\put(280,50){\line(0,10){80}}
			\put(260,50){\line(0,10){80}}
			\put(240,50){\line(0,10){100}}
			%
			\put(240,50){\line(10,0){100}}
			\put(240,130){\line(10,0){20}}
			\put(240,110){\line(10,0){40}}
			\put(240,90){\line(10,0){60}}
			\put(240,70){\line(10,0){140}}

			\thicklines
			\put(420,50){\line(0,10){20}}
			\put(380,70){\line(0,10){20}}
			\put(320,90){\line(0,10){20}}

			\put(420,50){\line(10,0){40}}
			\put(300,110){\line(10,0){20}}
			\put(380,70){\line(10,0){40}}
			\put(320,90){\line(10,0){60}}
			\color{blue}
			\put(340,50){\line(4,0){110}}
			\color{black}
			\color{blue}
			\put(340,50){\line(0,10){20}}
			\put(320,70){\line(10,0){20}}
			\put(320,70){\line(0,10){20}}
			\put(300,90){\line(10,0){20}}
			\put(300,90){\line(0,10){20}}
			
			\put(280,110){\line(0,10){20}}
			\put(260,130){\line(0,10){20}}
			
			\put(260,130){\line(10,0){20}}
			\put(240,150){\line(10,0){20}}
			\put(280,110){\line(10,0){20}}

			\color{black}
			\put(417,40){4}
			\put(397,40){3}
			\put(377,40){2}
			\put(357,40){1}
			\put(337,40){0}
			\put(315,40){-1}
			\put(295,40){-2}
			\put(275,40){-3}
			\put(255,40){-4}
			\put(235,40){-5}

		\end{picture}
		
		\caption{Graphs  of $S(k,0)$ (blue) and $S(k,t)$ (black) after eight jumps: four by the first particle, three by the second, and one by the third. The area of the differences $S(k,t)-S(k,0)$, $k\in \Bbb Z$, is the number of jumps by the process up to time $t$.\label{fig_Skt}}
	\end{figure}
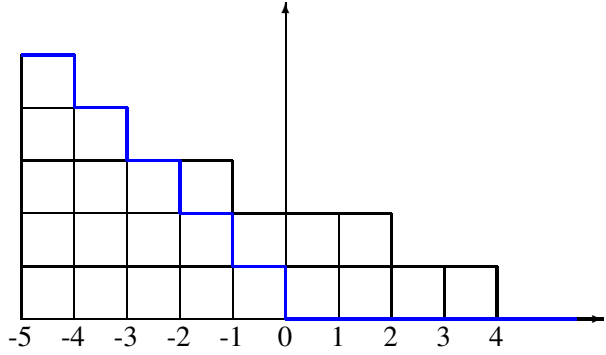

	Therefore, we need  to show that 
	\bel{main_lim}\lim_{t\to\infty} \int_\R\,\mathds E\,\left(\frac{S(\lf ut\rf,t)-S(\lf ut\rf ,0)}t\right)du=\frac16.
	\ee
	As was shown in \cite[Proposition~2]{r}
	for every $u\in\R$, $S(\lf ut\rf,t)/t$ converges almost surely and in $L_1$ to $h(u)$, defined by 
	\bel{h} h(u)=\left\{\begin{array}{ll}-u,& u<-1;\\ \frac14(1-u)^2,&-1\le u\le 1;\\0,&u>1.\end{array} \right.\ee
	
	To complete the proof, we will justify passing with the limit under the integral sign in \eqref{main_lim}. In order to do it, we confine ourselves to the integration over the positive half--line. To this end, first note that by symmetry of the particles and empty sites  (which we will refer to as 'holes'), for every $t\ge0$, we have
	\begin{equation}\label{XX}(X(j,t))_{j\in\Z}\stackrel d= (1-X(-j+1,t))_{j\in\Z},\end{equation}
	where the symbol \lq\lq$\stackrel d=$\rq\rq means equality in distribution. Further, if we let 
	\[S'(k,t):=\sum_{j<k}(1-X(j,t))\]
	be the number of holes to the left of $k$ then for all $t\ge0$ and $k\in\Z$
	\begin{equation}\label{SkS}S'(k,t)=k-1+S(k-1,t).\end{equation}
	This is because the equality holds at $t=0$ and when a particle (at a position $k$, say) jumps to the right, the number of holes to the left of $k+1$ increases by 1.

	Moreover, in view of \eqref{XX}, for every $t\ge 0$ we have
	\begin{equation}\label{SS} (S'(-k,t))_{k\in\mathbb Z}\stackrel{d}{=}(S(k+1,t))_{k\in\mathbb Z}. \end{equation} 
		Therefore, for any $u_0\ge0$
	\begin{align*}&\int_{-\infty}^{-u_0}\,\mathds E\,\frac{S(\lf ut\rf,t)-S(\lf ut\rf,0)}tdu\stackrel{\eqref{SkS}}{=}\int_{-\infty}^{-u_0}\,\mathds E\,\frac{S'(\lf ut\rf+1,t)+\lf ut\rf-\lf ut\rf}tdu\\&\quad=
		\int_{u_0}^\infty\,\mathds E\,\frac{S'(\lf -ut\rf+1,t)}tdu\stackrel{\eqref{SS}}{=}\int_{u_0}^{\infty}\,\mathds E\,\frac{S(\lf ut\rf+1,t)}tdu\\&\quad=\int_{u_0}^{\infty}\,\mathds E\,\frac{S(\lf ut\rf,t)}tdu-\int_{u_0}^\infty\,\mathds E\,\frac{X(\lf ut\rf+1,t)}tdu.
	\end{align*}
		Note that in the second line above, when we referred to \eqref{SS}, we first used at the left--hand side the identity $\lfloor -ut\rfloor+1=-\lfloor ut\rfloor$ for $ut\notin\mathbb Z$. 
	
	Similarly,
	$$\int_{-u_0}^0\,\mathds E\,\frac{S(\lf ut\rf,t)-S(\lf ut\rf,0)}tdu=\int_0^{u_0}\,\mathds E\,\frac{S(\lf ut\rf,t)}tdu-\int_0^{u_0}\,\mathds E\,\frac{X(\lf ut\rf+1,t)}tdu.$$

	Now, $X(\lf ut\rf+1,t)=1$ means that there is a particle at the position $\lf ut\rf+1$ at the time $t$. This, in turn, means that the first particle is to the right of $\lf ut\rf$ at the time $t$. 
	Since the first particle  jumps according to the Poisson process, we have
\[\int_{u_0}^\infty\,\mathds E\, X(\lf ut\rf+1,t)du\le \int_0^\infty\,\mathds P(N_t^{(1)}\ge\lf ut\rf)du\le 
1.
\]
Likewise
\[
\int_0^{u_0}\,\mathds E\, X(\lf ut\rf+1,t)du\le 
1.
\]
Hence, if for any $u_0>0$ we write
\bel{split}\int_\R\,\mathds E\,\frac{S(\lf ut\rf,t)-S(\lf ut\rf,0)}tdu=\left(\int_{|u|\le u_0}+\int_{|u|>u_0}\right)\left(\,\mathds E\,\frac{S(\lf ut\rf,t)-S(\lf ut\rf,0)}t\right)du\ee
then 
(note that $S(\lf ut \rf,0)=0$ for $ut>0$)
\begin{align}\label{split_small}\int_{|u|\le u_0}\,\mathds E\,\frac{S(\lf ut\rf,t)-S(\lf ut\rf,0)}tdu&=2\int_0^{u_0}\,\mathds E\,\frac{S(\lf ut\rf,t)}tdu+O(t^{-1}), \\
	\int_{|u|> u_0}\,\mathds E\,\frac{S(\lf ut\rf,t)-S(\lf ut\rf,0)}tdu&=2\int_{u_0}^\infty\,\mathds E\,\frac{S(\lf ut\rf,t)}tdu+O(t^{-1}), \mbox{as}\quad t\to\infty.\label{split_big}
\end{align} 
We first handle the integral  \eqref{split_big}.  Consider  $u>u_0$ with $u_0>0$ to be specified later. 

By Lemma~\ref{lem_exp_skt} the integrand on the right--hand side of \eqref{split_big} is bounded by $t^{\lf ut\rf}/\lf ut\rf!$. From Stirling's approximation
\[n!=\sqrt{2\pi n}\left(\frac ne\right)^n\left(1+O(n^{-1})\right)
\]  
it is clear that there exist a $u_0>1$ and $\alpha>0$ and $C>0$ such that this integrand is further bounded by 
\[
\frac{t^{\lf ut\rf}}{(\lf ut\rf)!}=\frac 1{\sqrt{2\pi\lf ut\rf}} 
\left(\frac{et}{\lf ut\rf}\right)^{\lf ut\rf}\left(1+O(\lf ut\rf^{-1}\right)\le C\frac{
	e^{-\alpha tu}}{\sqrt{ tu}} ,\]
uniformly over $u\ge u_0$ whenever $t$ is bounded away from zero. Consequently, 
\bel{bound_large_u}\int_{u_0}^\infty\,\mathds E\,\frac{S(\lf ut\rf,t)}tdu\le C\int_{u_0}^\infty
\frac{e^{-\alpha ut}}{\sqrt{tu}}du=o(1),\quad \mbox{as}\quad t\to\infty
\ee
which shows that the contribution of \eqref{split_big} is negligible.

We now turn our attention to \eqref{split_small}. 
The integrand on the right hand side of \eqref{split_small} is bounded by the integrable function $\II_{[0,u_0]}$ since by the monotonicity of $S(k,t)$ in $k$ and  Lemma~\ref{lem_exp_skt}
\[\mathds E\,S(\lf ut\rf,t)\le \mathds E\, S(0,t)\le \, t.
\]
Therefore, by the dominated convergence theorem
\[\lim_{t\to\infty}\int_{0}^{u_0}\,\mathds E\,\frac{S(\lf ut\rf,t)}tdu=\int_{0}^{u_0}\lim_{t\to\infty}\,\mathds E\frac{S(\lf ut\rf,t)}tdu
=\int_0^{u_0}h(u)du
=\frac14\int_{0}^1(1-u)^2du
=\frac1{12}.\]

Combining this with \eqref{split}--\eqref{split_big} and \eqref{bound_large_u} proves \eqref{main_lim}.
\endproof

\section{Number of active particles}

Let us denote by $a_t$ number of active particles  at time $t\ge 0$, i.e. 
\begin{equation}\label{repra}
	a_t=\sum_{k\in \mathbb Z}\,X(k,t)(1-X(k+1,t)).
\end{equation}

\begin{theorem}\label{thm_act_part}
	We have
	\[\lim_{t\to\infty}\frac{\mathds E\,a_t}t=\frac13.\]
\end{theorem}
\proof
We write,
\[\mathds E\, a_t=\,\mathds E\,\sum_{k\in\Z}X(k,t)(1-X(k+1,t))=\,\mathds E\,\left(\sum_{|k|< k_t}+\sum_{|k|\ge k_t}\right)X(k,t)(1-X(k+1,t))
\]
for $k_t$ which will be  chosen later.

Consider first
\[\mathds E\,\sum_{k\ge k_t}X(k,t)(1-X(k+1,t))\le \,\mathds E\,\sum_{k\ge k_t}X(k,t)=\mathds E\,S(k_t-1,t)
\le \frac{t^{k_t}}{(k_t-1)!}
\]
where the last inequality is by Lemma~\ref{lem_exp_skt}. 
Choosing $k_t\sim \beta t$ for sufficiently large (but fixed) $\beta>1$ and using Stirling approximation as in an earlier argument 
we conclude that the quantity on the right vanishes as $t\to\infty$, that is
\[\,\mathds E\,\sum_{k\ge k_t}X(k,t)(1-X(k+1,t))=o(1),\quad t\to\infty.\]

Reversing the roles of $X(k,\, \cdot\, )$ and $1-X(k+1,\, \cdot\,)$ and using the migration of the first hole rather than the first particle we similarly get that 
\[\mathds E\,\sum_{k\le-k_t}X(k,t)(1-X(k+1,t))=o(1),\quad t\to\infty.\]
It remains to asymptotically evaluate
\[
\frac1t\E\,\sum_{|k|< k_t}X(k,t)(1-X(k+1,t))=\frac1t\E\, \sum_{|k|< k_t}X(k,t)-\frac1t\E\,\sum_{|k|< k_t}X(k,t)X(k+1,t).
\]
By \cite[Theorem~1]{r}, as $t\to\infty$
\bel{as}\frac1t\sum_{ut<k<v t}X(k,t)\stackrel{a.s.}{\longrightarrow} \int_u^vf(w)dw,\ee
where 
\begin{equation}\label{ef}f(w)=\left\{\begin{array}{ll}\frac12(1-w),&  |w|\le1;\\1,&w<-1;\\0,&w>1.\end{array}\right.\end{equation}
By the bounded convergence theorem, the expected value of the left--hand side of \eqref{as} converges as well. Therefore,
\bel{int_f}
\frac1{t}\,\mathds E\,\sum_{|k|<k_t}X(k,t)=\frac1{t}\,\mathds E\,\sum_{-\beta t<k<\beta t}X(k,t)
\longrightarrow\int_{-\beta}^\beta f(w)dw
\ee
as $t\to\infty$. 

Further, Rost proved that (see \cite[(1) and (3) of Section~4]{r}; note the ambiguity of display enumeration across the sections in \cite{r}), in our notation
\begin{equation}\label{ros1}
	\lim_{t\to\infty}\,\mathds E\,X(\lf ut\rf,t)X(\lf ut\rf+1,t)\to f^2(u),
\end{equation}
where $f$ is  given in \eqref{ef}.  Therefore, using $k_t\sim \beta t$ again we have
\begin{align}&\frac1{t}\,\mathds E\,\sum_{|k|< k_t}X(k,t)X(k+1,t)=\frac1{t}\,\mathds E\,\int_{-\beta t}^{\beta t}X(\lf v\rf,t)X(\lf v\rf+1,t)dv\nonumber\\&\quad
	=\int_{-\beta }^{\beta }\,\mathds E\,X(\lf ut\rf,t)X(\lf ut\rf+1,t)du\longrightarrow\int_{-\beta}^\beta f^2(u)du\label{Rostf2}
\end{align}
as $t\to\infty$. 
We can exchange the order of the integrals by non--negativity and we can take the limit inside by the bounded convergence theorem since the integrands are bounded by $\II_{[-\beta,\beta]}(u)$.

Combining this with \eqref{int_f}  and recalling that $\beta>1$  we obtain
\[\frac1{t}\,\mathds E\,\sum_{|k|< k_t}X(k,t)(1-X(k+1,t))\longrightarrow \int_{-\beta}^{\beta}(f(u)-f^2(u))du=\int_{-1}^{1}\left(\frac{1-u}2-\frac{(1-u)^2}4\right)du=\frac13.
\]
\endproof

\section{A connection}
As we mentioned earlier, it is natural to expect  a connection between the speed with which the process moves at a given time and the number of active particles at that time. 
Since we have not seen any formal statement in the literature linking these two quantities we include a statement and  its proof here. 


\begin{theorem}\label{thm_active}Let $g(t)=\mathds E\,P_t$, $t\ge 0$. Then
	\begin{equation}\label{derivative}
		g'(t)=\E\,a_t,\quad t>0.
	\end{equation}	
\end{theorem}
\begin{proof}
	We first consider the case of the right derivative. That is, we  take $h>0$.
	Note that by \eqref{pt}
	\[g(t+h)-g(t)=\mathds E\,(P_{t+h}-P_t)=\mathds E\,\sum_{k\in\Z}(S(k,t+h)-S(k,t))
	.\]
	Set
	\[\D_{k,t,h}:=S(k,t+h)-S(k,t).\]
	Then,
	\bel{exp_delta}g(t+h)-g(t)=\sum_{k\in\Z}\,\mathds E\,\D_{k,t,h}=\sum_{k\in\Z}\sum_{m\ge1}\,\mathds P(\D_{k,t,h}\ge m).\ee
	To evaluate these probabilities, for a given  $k$ and $t$  consider the following events (where we suppress the dependence on $t$ in the notation):
	\begin{align*}
		A_k&=\left\{\dots \stackrel k\bullet \circ \dots \right\} \quad\mbox{(there is a particle on the $k$th position but not on the $(k+1)$st at time $t$})
		\\B_k&=\left\{\dots \stackrel k\circ\bullet\dots \right\} \quad\mbox{(no particle on the $k$th position but there is one on the $(k+1)$st at time $t$})
		\\C_k&=\left\{\dots\stackrel k\bullet\bullet\dots \right\}\quad\mbox{(both the $k$th and the $(k+1)$st positions are occupied at time $t$})
		\\D_k&=\left\{\dots \stackrel k\circ\circ\dots\right\}\quad\mbox{(both the $k$th and the $(k+1)$st positions are empty at time $t$}).
	\end{align*}
	Then, the probability on the right hand side of \eqref{exp_delta} is the sum  
	\bel{four}\,\mathds P(\D_{k,t,h}\ge m, A_k)+\,\mathds P(\D_{k,t,h}\ge m, B_k)+\,\mathds P(\D_{k,t,h}\ge m, C_k)+\,\mathds P(\D_{k,t,h}\ge m,D_k).\ee
	We will show that the dominant contribution to the right hand side of \eqref{exp_delta} comes from the case
	\[\sum_{k\in\Z}\,\mathds P(\D_{k,t,h}\ge1,A_k).\]
	We have
	\[\mathds P( \D_{k,t,h}\ge1,A_k)=\mathds P(\D_{k,t,h}\ge1|A_k)\,\mathds P(A_k)=(1-\,\mathds P(\D_{k,t,h}=0|A_k))\,\mathds P(A_k).\]
	Given $A_k$, $\D_{k,t,h}=0$ means that the particle that is on the $k$th position at a time $t$ does not move in the time interval  $[t,t+h]$ which happens with probability $e^{-h}$.
	Thus,
	\bel{Akm=1}\sum_{k\in\Z}\,\mathds P(\D_{k,t,h}\ge1,A_k)=(1-e^{-h})\sum_{k\in\Z}\,\mathds P(A_k)=(1-e^{-h})\,\mathds E\, \sum_{k\in\Z}\II_{A_k}=(1-e^{-h})\E\, a_t\ee
	which asymptotically as $h\to0$ is $h\,\mathds E\, a_t$. Thus, the theorem will be proved once we show that all other terms are of lower order as $h\to0$. 
		
	For $m\ge2$ and $k\ge0$ the event $\{ \D_{k,t,h}\ge m,A_k\}$ means that at least $m$ particles crossed position $k$ between $t$ and $t+h$ and that the first particle made at least $k$ jumps by the  time $t$ (at time $t$ there is a particle on the position $k$, so it must have made at least $k$ jumps, and hence the first one must have made at least that many jumps). For these things to happen, the first particle must have made at least $k$ jumps by the time $t$, the particle at the position $k$ at the time $t$ must have jumped between $t$ and $t+h$ and the  clocks of  each of  the $m-1$  particles  closest to $k$ from the left must have  rung at least once each 
	between $t$ and $t+h$ (some of their jumps may have been blocked, but the clocks must have rung). Since all these events are independent, we see that 
	\bel{kge0}\P(\D_{k,t,h}\ge m,A_k)\le (1-e^{-h})^m\P(N_t^{(1)}\ge k).
	\ee
	
	If $k<0$, $\{ \D_{k,t,h}\ge m,A_k\}$ means that at least $m$ particles crossed position $k$ between $t$ and $t+h$ and that the first hole made at least $k$ moves by the  time $t$. Throughout the remainder of the  proof by the first hole  we mean  the leftmost hole, i.e.  the hole positioned at $1$ when $t=0$; the number of moves by that hole up to time $t$ is denoted by $H_t^{(1)}$  and, in view of the particle-hole duality, $(H_t^{(1)})$ is a Poisson process.
	Consequently,  arguing as for \eqref{kge0}, we see that 
	\[\P(\D_{k,t,h}\ge m,A_k)\le (1-e^{-h})^m\P(H_t^{(1)}\ge -k).
	\]
	
	Clearly, $N_t^{(1)}\stackrel{d}{=}H_t^{(1)}\stackrel{d}{=}\PP_t$, where  $\PP_t$ is a $\rm{Poisson}(t)$  random variable. Thus it follows from this and \eqref{kge0} that 
	\begin{align}\nonumber
		\sum_{m\ge2}\sum_{k\in\Z}\P(\D_{k,t,h}\ge m,A_k)
		&\le(1-e^{-h})^2e^h\left(\sum_{k\ge0}\P(N_t^{(1)}\ge k)+\sum_{k\ge1}\P(H_t^{(1)}\ge k)\right)&\\
		\label{Akm_large}\\&=(1-e^{-h})^2e^h(1+2\E\,\PP_t)
		=(1-e^{-h})^2e^h(1+2t),\nonumber\end{align}
	which is of order $h^2$ as $h\to0$.

	The argument for $B_k$ 
	is virtually the same. Here, for at least $m$ crosses of $k$ between $t$ and $t+h$, the clock of the particle in position $k+1$  must ring in this interval and the clocks of the $m$ preceding particles must ring between $t$ and $t+h$. Also, since there is a particle in position $k+1$ at time $t$ (a hole in position $k$ if $k<0$), the first particle   must have made at least $k+1$ jumps by the time $t$ (or the first hole must have made at least $-k+1$ moves if $k<0$). Using independence  and summing over $m\ge1$ and $k\in\Z$ we get 
	\begin{align}\nonumber\sum_{m\ge1}\sum_{k\in\Z}\P(\D_{k,t,h},B_k)&\le
		\sum_{m\ge1}(1-e^{-h})^{m+1} \left(\sum_{k\ge0} \P(N_t^{(1)}\ge k+1)+\sum_{k\ge1}\P(H_t^{(1)}\ge k+1)\right)\\&\label{Bk}\\&\le2(1-e^{-h})^2e^ht.\nonumber\end{align}
	
	The argument for $C_k$ and $D_k$ is a bit more delicate since  we do not know where the first particle (in the case of  $D_k$) or the  first  hole (in the case of $C_k$) at time $t$ is.  
	
	Consider first $k<0$ for $D_k$.  For  $m$ crosses of $k$ between $t$ and $t+h$, each of the clocks of the $m$ particles closest to $k$ from the left must  ring at least twice between $t$ and $t+h$. Further, since the first hole is in position at most $k$ at the time $t$,   $H_t^{(1)}\ge -k+1$ on $D_k$.
	Therefore, for $k<0$
	\bel{Dkminus}\P(\D_{k,t,h}\ge m,D_k)\le (1-e^{-h})^{2m}\P(H_t^{(1)}\ge -k+1).
	\ee
	Consider now $k\ge0$. 
	We split $D_k$ according to whether the position of the first particle at the time $t$ is to the left  or to the right of $k$:
	\[D_k=D_k^{>}\cup D_k^{<},\quad\mbox{where}\quad D_k^{>}=D_k\cap\{N_t^{(1)}\ge k+2\},\quad D_k^{<}=D_k\cap\{N_t^{(1)}< k\}.\]
	Thus, by independence
	\bel{Dk>}\P(\D_{k,t,h}\ge m, D_k^{>})\le (1-e^{-h})^{2m}\P(N_t^{(1)}\ge k+2).
	\ee   
	It follows from \eqref{Dkminus} and \eqref{Dk>} that 
	\[\sum_{m\ge1}\sum_{k\in\Z}\P(\D_{k,t,h}\ge m,D_k)\le \frac{(1-e^{-h})^2e^h}{2-e^{-h}}(2\E\,\PP_t-1)+\sum_{m\ge1}\sum_{k\ge0}\P(\D_{k,t,h}\ge m,D_k^<).
	\]
	The first term on the right--hand side is $O(h^2)$ so it remains to show that 
	\bel{Dk<}\sum_{m\ge1}\sum_{k\ge0}\P(\D_{k,t,h}\ge m,D_k^<)=o(h).
	\ee
	First note that the inner sum is in fact over  $k\ge 1$ (as the leftmost position, the first particle can ever be at, is 0). At time $t$ the first particle can be at any position $i$, $0\le i<k$,  and must be to the right of position $k+m-1$ at time $t+h$ (to make room to the right of  position  $k$ for at least $m-1$ particles). In other words, the first particle makes $i$ jumps up to time $t$ and at least  $k+m-i$ additional jumps between $t$ and $t+h$. Probability of this event is 
	\[\sum_{i=0}^{k-1}\P(N_t^{(1)}=i, N^{(1)}_{t+h}-N^{(1)}_t\ge k+m-i)=\sum_{i=0}^{k-1}\P(N_t^{(1)}=i)\P(\PP_h\ge k+m-i)
	\]
	where we have used the independence of $N_t^{(1)}$ and $N_{t+h}^{(1)}$ and where $\PP_h$ is a  $\mathrm{Poisson}(h)$ random variable. We split the sum on the right as
	\begin{align}\nonumber&\left(\sum_{i=0}^{\lf k/2\rf}+\sum_{i=\lf k/2\rf+1}^{k-1}\right)\P(N^{(1)}_t=i)\P(\PP_h\ge k+m-i)\\&\quad \nonumber
		\le \sum_{i=0}^{\lf k/2\rf}\P(N^{(1)}_t=i)\P(\PP_h\ge k+m-\lf k/2\rf)+
		\sum_{i=\lf k/2\rf+1}^{k-1}\P(N^{(1)}_t=i)\P(\PP_h\ge m+1)\\&\le \P(\PP_h\ge k+m-\lf k/2\rf)+\P(N^{(1)}_t\ge\lf k/2\rf+1)\P(\PP_h\ge m+1).\label{2sums}
	\end{align}
	Summing the above expressions over $k,m\ge1$ and using 
	\[\P(\PP_h\ge r)\le\sum_{i\ge r}h^i=\frac{h^r}{1-h}\]
	and $\lf k/2\rf+1\ge k/2$ we see that the 
		double 
	sum of the second terms in \eqref{2sums} is bounded by 
	\[\sum_{k\ge1}\P(N_t^{(1)}\ge k/2)\sum_{m\ge1}\frac{h^{m+1}}{1-h}=2\E\,N_t^{(1)}\frac{h^2}{(1-h)^2}=2t\frac{h^2}{(1-h)^2}.\]

	For the  double  sum of the first terms  we get  
		(recall that necessarily $k\ge 1$) 
		\begin{align*}\sum_{k\ge 1}\,\sum_{m\ge 1}&\,\P(\PP_h\ge k+m-\lf k/2\rf)\le \sum_{k\ge 1}\,\sum_{m\ge 1}\, \tfrac{h^{k+m-\lf k/2\rf}}{1-h}=\tfrac 1{1-h}\left(\sum_{m\ge1}h^m\right)\left(\sum_{k\ge1}h^{k-\lf k/2\rf}\right)\\&=\frac h{(1-h)^2}\left(\sum_{j\ge0}h^{2j+1-j}+\sum_{j\ge1}h^{2j-j}\right)=2\frac{ h^2}{(1-h)^3}.
	\end{align*}
	This shows \eqref{Dk<} and combined with the earlier estimates gives
	\bel{Dk}\sum_{m\ge1}\sum_{k\in\Z}\P(\D_{k,t,h}\ge m,D_k)=O(h^2),\quad h\to0.
	\ee
	
	It remains to consider the $C_k$. If $k\ge0$ then the same argument as for the $B_k$ (the first particle must have made at least $k+1$ jumps by the time $t$ since there is a particle in position $k+1$; for at least $m$ crosses of $k$ between $t$ and $t+h$ the clock of the particle in position $k+1$ at a time $t$ and the clocks of $m$ particles to the left of $k$ must have rung between $t$ and $t+h$) yields 
	\bel{Ck>0}\sum_{m\ge1}\sum_{k\ge0}\P(\D_{k,t,h}\ge m,C_k)\le\sum_{m\ge1}\sum_{k\ge0}(1-e^{-h})^{m+1}\P(\PP_t\ge k+1)=
	O(h^2)\E\PP_t=O(h^2)
	.\ee
	When $k<0$ we split $C_k=C_k^<\cup C_k^>$ according to whether the position of the first hole at the time $t$ is to the left of $k$  or to the right of $k+1$. Note that on $C_k^<$ we have $H_t^{(1)}\ge -k+2$ (as the first hole have moved from $1$ to the left of $k$ by the time $t$) and the clocks of at least $m+1$ specific particles rung between $t$ and $t+h$. Therefore, 
	\bel{Ck<}\sum_{m\ge1}\sum_{k<0}\P(\D_{k,t,h}\ge m,C_k^<)\le\sum_{m\ge1}(1-e^{-h})^{m+1}\sum_{k\ge1}\P(\PP_t\ge k+2)=O(h^2).\ee
	Finally, consider  $C_k^>$. Let the position of the first hole at the time $t$ be $k+2+i$ for some $0\le i\le-k-1$ (the upper bound follows from  $k+i+2\le1$ as $1$ is the rightmost position the first hole can  ever be at). That means that up to time $t$, the first hole made $-k-i-1$ moves. It needs to make additional moves between $t$ and $t+h$. Namely, it needs to move to position $k+1$ (which takes $i+1$ moves) and then it needs to make at least $m$ additional moves to facilitate that many  crossings of $k$; a total of at least $m+i+1$ moves between $t$ and $t+h$.
	Thus, 
	\begin{align*}\P(\D_{k,t,h}\ge m,C_k^{>})&\le\sum_{i=0}^{-k-1}\P(H_t^{(1)}=-k-i-1,H_{t+h}^{(1)}-H_t^{(1)}\ge m+i+1)\\&=
		\sum_{i=0}^{j-1}\P(\PP_t=j-i-1)\P(\PP_h\ge m+i+1)
		.\end{align*}
	Using arguments parallel to those used for \eqref{2sums} and \eqref{Dk} we infer that 
	\[\sum_{m\ge1}\sum_{k<0}\P(\D_{k,t,h}\ge m,C_k^>)=O(h^2).\]
	Combining the last estimate with \eqref{Ck<} and \eqref{Ck>0}  we conclude that 
	\bel{Ck}\sum_{m\ge1}\sum_{k\in\Z}\P(\D_{k,t,h}\ge m,C_k)=O(h^2),\quad h\to0.
	\ee 
	
	Now, \eqref{Akm_large}, \eqref{Bk}, \eqref{Dk} and \eqref{Ck} imply that \eqref{Akm=1} is asymptotically dominant among   expressions \eqref{four}. Thus, referring to \eqref{exp_delta}, we obtain that 
	\[\lim_{h\to0+}\frac{g(t+h)-g(t)}h=\E\,a_t.\]
	The argument for $h\to0^-$ is basically the same with $t-h$ playing the role of $t$ and $t$ playing the role of $t+h$.  This completes the proof.
\end{proof}

\section{Remarks}

\subsection{Jumps of the rightmost particles}

For $r\ge1$ let $N^{(r)}_t$ be the number of jumps of the $r$th particle up to time $t$. Then 
\[N_t^{(r)}\le N_t^{(r-1)}\le\dots\le N_t^{(1)}\]
and since the number of jumps of the first particle by the time $t$  is governed by a  Poisson process 
we conclude that for every fixed $r\ge 1$
\[\E\ N_t^{(r)}\le t.\]
Nonetheless, while the $r$th particle moves slower than the first one,  the next statement asserts that asymptotically it still moves at a rate $t$ as $t\to\infty$.

\begin{proposition}
For any $r\ge 1$ we have 
\[\lim_{t\to\infty}\frac{\E\ N_t^{(r)}}t=1.\]
\end{proposition}
\proof
By the earlier considerations
\[\limsup_{t\to\infty}\frac{\E\ N_t^{(r)}}t\le1\]
so it remains to prove that 
\bel{liminf}\liminf_{t\to\infty}\frac{\E\ N_t^{(r)}}t=1.\ee
To this end, we first note that 
\[N^{(r)}_t=r-1+\color{black}\inf\{k: S(k,t) < r\}\]
(in particular, for the first particle    $N^{(1)}_t=\inf\{k: S(k,t)<1\}
$).
Consequently, 
\[\{N^{(r)}_t\ge  k\}=\{S(k-r,t)\ge  r)\}\]
and it follows that
\[\E N_t^{(r)}=\sum_{k\ge 1} \P(N_t^{(r)}\ge k)=
\sum_{k\ge 1} \P(S(k-r,t)\ge r)\ge
\sum_{k\ge 0} \P(S(k,t)\ge r)
.\]

Pick any $0<\eps<1$.  Then 
\begin{align}\nonumber\frac{\E\ N_t^{(r)}}t&\ge \frac1t\sum_{k\ge0}\P(S(k,t)\ge r)=\frac1t\E\,\sum_{k\ge0}\II_{S(k,t)\ge r}\ge\frac1t\sup_{y>0}y\P(\sum_{k\ge0}\II_{S(k,t)\ge r}\ge y)
\\&\label{lbN} \\&\ge
\frac{(1-\eps)t+1}t\,\P(\sum_{k\ge0}\II_{S(k,t)\ge r}\ge (1-\eps)t+1)\ge(1-\eps)\P(\sum_{k\ge0}\II_{S(k,t)\ge r}\ge\lf (1-\eps)t\rf+1).\nonumber
\end{align}
Since $S(k,t)$ is decreasing in $k$,  for $t>0$  we have 
$\sum_{k\ge0}\II_{S(k,t)\ge r}\ge\lf (1-\eps)t\rf+1)$ if and only if $S(\lf (1-\eps)t\rf,t)\ge r$. 
Thus,
\bel{lbP}\P(\sum_{k\ge0}\II_{S(k,t)\ge r}\ge\lf (1-\eps)t\rf+1)=\P(S(\lf(1-\eps)t\rf,t)\ge r)=\P\left(\frac{S(\lf(1-\eps)t\rf,t)}t\ge \frac rt\right).
\ee
Let $t_0=8r/\eps^2$. For $t\ge t_0$, we have $r/t\le\eps^2/8$ and thus for such $t$
\bel{fin_lb}\P\left(\frac{S(\lf(1-\eps)t\rf,t)}t\ge \frac rt\right)\ge \P\left(\frac{S(\lf(1-\eps)t\rf,t)}t\ge \frac{\eps^2}8\right).\ee
As $t\to\infty$, 
\[\frac{S(\lf(1-\eps)t\rf,t)}t\stackrel{a.s.}{\longrightarrow} h(1-\eps)=\frac14\eps^2\]
where we have used the expression for $h(u)$ given in \eqref{h}.
Therefore,
\[\P\left(\frac{S(\lf(1-\eps)t\rf,t)}t\ge \frac{\eps^2}8\right)\longrightarrow1,\quad t\to\infty.\]
Combining this with \eqref{lbN}--\eqref{fin_lb} yields
\[\liminf_{t>0}\frac{\E\ N^{(r)}_t}t\ge 1-\eps.\]
Since $\eps$ was arbitrary, this proves \eqref{liminf} and  the Proposition.
\endproof

\subsection{Slow decrease 
}
Although we furnished independent proofs of Theorem~\ref{pole} and Theorem~\ref{thm_act_part} one might expect that Theorem~\ref{thm_active} would provide a way to deduce one from another. In fact, were the statements of Theorems~\ref{pole} and~\ref{thm_act_part} true for all $t$ rather then only in the limits, Theorem~\ref{thm_act_part} would have been an obvious corollary to Theorems~\ref{pole} and~\ref{thm_active}. In general, this is not true and one would need additional information about $\E\ a_t/t$ to deduce Theorem~\ref{thm_act_part} from  Theorems~\ref{pole} and~\ref{thm_active}. One such condition is a slow decrease. 
Recall (see e.g. \cite[Definition~p.~41; it seems that a 'slowing' in that definition should read 'slowly']{bgt} or \cite[(6.2.2),  (6.2.3), Section~6.2, p.~124]{h})  that a function $f$ is slowly decreasing 
if  
\bel{slo_mo}
\lim_{\la\searrow1}\liminf_{x\to\infty}\inf_{\al\in [1,\lambda]}\left(f(\al x)-f(x)\right)\ge0.
\ee
(Note that since the infimum in \eqref{slo_mo} is non--positive, \eqref{slo_mo} is equivalent to the quantity on the left being equal to zero.)  For a slowly decreasing functions $f$ one has (see \cite[Item~15~Section~1.11, p.~59]{bgt} or \cite[footnote to Theorem~68, p.~126]{h})
\[
\frac1x\int_0^x\,f(t)\,dt\longrightarrow c,\quad x\to\infty\quad\implies f(x)\to c,\quad x\to\infty.
\]
Knowing that the function $f(t)= \E\,a_t/t$, $t\ge0$, is slowly decreasing would allow one to derive Theorem~\ref{thm_act_part} from  Theorems~\ref{pole} and~\ref{thm_active}.
Since we provided a direct proof of  Theorem~\ref{thm_act_part} we  only sketch a proof  that  $\E\, a_t/t$ is, indeed, slowly decreasing:
\begin{proposition} 
The function
\[t \longrightarrow \frac{\E\,a_t}t,\quad t\ge0\] is slowly decreasing.
\end{proposition}

\begin{proof} (Sketch)
Since  
\[\frac{\E\, a_{\al t}-\al\E\, a_t}
{\al t}
=\frac{\E\, a_{\al t}-\E\, a_t}{\al t}
-\frac{(\al-1)\E\, a_t}{\al t},
\]
by superadditivity of  $\inf$ and $\liminf$, the left--hand side of 
\eqref{slo_mo} for our function is bounded below by
\[\lim_{\la\searrow1}\liminf_{t\to\infty}\inf_{\al\in [1,\lambda]}
\frac{\E\, a_{\al t}-\E\, a_t}{\al t}
-\lim_{\la\searrow1}\limsup_{t\to\infty}\sup_{\al\in [1,\lambda]}\frac{(\al-1)\E\,a_t}{\al t}.\]
Since $1\le a_t\le N_t^{(1)}$ we have $1\le\E\, a_t\le t$ and hence, the second term vanishes as $\la\searrow1$. 
We need to show that
\bel{slo_main}\lim_{\la\searrow1}\liminf_{t\to\infty}\inf_{\al\in [1,\lambda]}
\frac{\E\, a_{\al t}-\E\, a_t}{\al t}=0.
\ee
For $t>0$ and  $s=(\al-1) t$ 
\begin{align*}\E\,(a_{t+s}-a_t)&
	=
	\E\,\left(\sum_{|k|< k_t}+\sum_{|k|\ge k_t}\right)\Big(X(k,t+s)(1-X(k+1,t+s))-X(k,t)(1-X(k+1,t))\Big).
\end{align*}

Similarly as in the proof of Theorem \ref{thm_act_part} we conclude that the sum over $|k|\ge k_t\sim\beta t$ vanishes as $t\to\infty$.
The sum over $|k|<k_t$  we split into \color{black}
\begin{align*}
	& \E\,\sum_{|k|< k_t}\Big(X(k,t+s)-X(k,t)\Big)+\E\,\sum_{|k|< k_t}\Big(X(k,t)X(k+1,t)-X(k,t+s)X(k+1,t+s)\Big).
\end{align*}
Again referring to the proof of Theorem \ref{thm_act_part}, we use \eqref{int_f} and \eqref{Rostf2} to derive the limits
\begin{align*}&
	\tfrac1{\al t}\E\,\sum_{|k|<k_t}\Big(X(k,t+s)-X(k,t)\Big)
		\to\int_{-\frac\beta\al}^{\frac\beta\al}f(w)dw-\tfrac1{\al}\int_{-\beta}^\beta f(w)dw
\end{align*}
and 
\[\tfrac1{\alpha t}\,\E\,\sum_{|k|\le k_t}\Big(X(k,t)X(k+1,t)-X(k,t+s)X(k+1,t+s)\Big)\to \tfrac1\al\int_{-\beta}^\beta f^2(w)dw -\int_{-\frac\beta\al}^{\frac\beta\al}f^2(w)dw
\]
as $t\to\infty$. Since $1<\al<\la\searrow1$ the right--hand sides above  converge to zero, and thus \eqref{slo_main} follows. 
\end{proof}
\color{black}

{\bf Acknowledgments.}  
This work was carried out during P.~H.'s visit to Warsaw University of Technology and was supported in part by the Ko\'sciuszko Foundation through its Research Exchange to Poland program. J.~W. was supported in part by
 National Science Center Poland [project no.  2023/51/B/ST1/01535].

 We would like to warmly thank Dominik Schmid for generously sharing with us his intuitions, for all inspiring discussions, as well as for a detailed literature guidance.

\end{document}